\documentclass[letterpaper,11pt]{article}

\usepackage[english]{babel}
\usepackage[utf8]{inputenc}
\usepackage{palatino}
\usepackage{amsmath, amssymb, amsthm, amsfonts}
\usepackage{graphicx}
\usepackage[colorinlistoftodos,textsize=scriptsize]{todonotes}
\usepackage{fullpage}
\usepackage[compact]{titlesec}
\usepackage[margin=0.5in]{caption}
\usepackage{enumitem}
\usepackage{hyperref}
\usepackage{url}
\usepackage{pgf,tikz}
\usepackage[all]{xy}
\usepackage{tikz-cd}
\usetikzlibrary{calc}
\usetikzlibrary{decorations.pathreplacing}
\usetikzlibrary{arrows}
\usepackage{mathrsfs}  
\usepackage{blkarray}
\usepackage{cite}
\usepackage{caption}
\usepackage{subcaption}
\usepackage[capitalise]{cleveref}
\usepackage{xspace}
\theoremstyle{plain}
\newtheorem{theorem}{Theorem}[section]

\newtheorem{lemma}[theorem]{Lemma}

\theoremstyle{definition}
\newtheorem{example}[theorem]{Example}

\theoremstyle{definition}

\DeclareMathOperator{\im}{\mathrm im}

\newcommand{\Z}{\mathbb{Z}}

\newcommand{\ZCat}{\mathbf{Z}}

\newcommand{\ZZCat}{\mathbf{ZZ}}

\newcommand{\idf}[1]{\emph{#1}}

\newcommand{\J}{{\mathcal J}}

\newcommand{\vect}{\mathbf{Vec}}

\date{}

\title{Interval Decomposition of Infinite Zigzag Persistence Modules}

\author{Magnus Bakke Botnan}

\begin{document}
\nocite{*}
\maketitle

\begin{abstract}
We show that every infinite zigzag persistence module decomposes into a direct sum of interval persistence modules. 
\end{abstract}
 
\section*{Introduction} A \emph{discrete persistence module} is a functor $M\colon \ZCat \to \vect$ where $\ZCat$ is the integers viewed as a poset category and $\vect$ is the category of finite dimensional vector spaces over some fixed field $\mathbb{F}$. It was proved by Webb\cite{webb} that a discrete persistence module admits a decomposition into a direct sum of \emph{interval persistence modules}. This was later generalized by Crawley-Boevey\cite{wcb} to persistence modules indexed over the category of real numbers. Another type of persistence module is the \emph{zigzag} persistence module considered by Carlsson and de Silva\cite{carlsson}. Such persistence modules also decompose into intervals, a fact well-known to representation theorists. In this note we generalize this result to infinite zigzags, or, in the language of representation theory, to locally finite representations of $A_\infty^\infty$ with arbitrarily ordered arrows. 

The fact that a locally finite dimensional representation of $A_\infty^\infty$ admits a direct sum decomposition into interval summands is known\cite{paq}, but to the best of the author's knowledge, the result is not present in the literature. Moreover, it should be emphasized that \cref{prop:decomp} appears in more general form in Section 6 (Covering Theory) of Ringel's Izmir Notes\cite{ringel}. 

Working with infinite zigzags proved itself convenient in an ongoing project but there was an apparent lack of citable sources on the interval decomposition of such zigzags. The author hopes this note fills that gap. Moreover, an immediate consequence of this approach is a new proof for the interval decomposition of discrete persistence modules. 

\paragraph{Acknowledgements} The author wishes to thank Jeremy Cochoy, Michael Lesnick, Steve Oudot and Johan Steen for valuable feedback.

\section{Zigzag Persistence Modules}
A \idf{zigzag persistence module} is a sequence of vector spaces and linear maps indexed by the integers 
\[V\colon \cdots \leftrightarrow V_{-1} \leftrightarrow V_0 \leftrightarrow V_1 \leftrightarrow \cdots \]
where $\leftrightarrow$ denotes an arrow of type $\leftarrow$ or $\rightarrow$.  This is a generalization of discrete persistence modules for which arrows point in the same direction. In this note we will restrict ourselves to zigzags persistence modules of the form  
\[ V\colon \cdots \rightarrow V_{-1} \leftarrow V_0 \rightarrow V_{1} \leftarrow \cdots, \]
i.e. where we have sinks at all odd numbers and sources at even numbers. Any other zigzag persistence module can be understood from such a zigzag by adding appropriate isomorphisms. 

In the language of category theory a zigzag persistence module is a functor $V\colon \ZZCat \to \vect$ where $\ZZCat$ is the category with objects the integers $\mathbb{Z}$, together with morphisms $i\to i-1$ and $i\to i+1$ for all even numbers $i$. We shall denote the morphisms $V(i\to i-1)$ and $V(i\to i+1)$ by $g_i$ and $f_i$, respectively.  For integers $s \leq t$ we define the \emph{restriction of $V$ to $[s,t]$} to be the persistence module $V|_{[s,t]}\colon \ZZCat|_{[s,t]}\to \vect$ where $\ZZCat|_{[s,t]}$ is the full subcategory of $\ZZCat$ with objects $\{ i \colon s\leq i\leq t\}$. A zigzag persistence module indexed by $\ZZCat|_{[s,t]}$ is \emph{finite}. 

 For $a\leq b\in \Z\cup \{\pm \infty\}$ define an \idf{interval (zigzag) persistence module} $I^{[a,b]}\colon \ZZCat \to \vect$ on objects by 
\[
I^{[a,b]}_i = \begin{cases} \mathbb{F} & \text{if $a\leq i\leq b$}\\ 0 & \text{otherwise}\end{cases}
\]
and which assigns the identity morphism to any morphism connecting two non-zero vector spaces. Note that we have adopted the convention $-\infty < i < +\infty$ for all $i\in \Z$. 

For zigzag persistence modules $U, W\colon \ZZCat\to \vect$, define their \idf{direct sum} $U\oplus W$ to be the persistence module defined on objects by $(U\oplus W)_i = U_i\oplus W_i$ and on morphisms by $(U\oplus W)(\alpha) = U(\alpha)\oplus W(\alpha)$. We say that $V$ is \idf{decomposable} if there exist non-zero $U,W$ such that $V\cong U\oplus W$. If no such decomposition exists then $V$ is \idf{indecomposable}. 

A weaker form of indecomposability is indecomposability over an interval. Let $s\leq t$ be integers; $0\neq W\colon \ZZCat \to \vect$ is \idf{$[s,t]$-indecomposable} if for any decomposition $W= W^1\oplus W^2$, then either $W^1_i = 0$ for all $i\in [s,t]$, or $W^2_i = 0$ for all $i\in [s,t]$. Moreover, an \idf{$[s,t]$-decomposition} of $V$ is a decomposition $V=\bigoplus_{j\in\J} W^{j}$ such that $W^{j}$ is $[s,t]$-indecomposable for all $j\in\J$. It is not hard to see that such a decomposition exists for every $[s,t]$: if $V$ is $[s,t]$-indecomposable then we are done. Otherwise, decompose $V$ and inductively choose a $[s,t]$-decomposition for each of its summands. Since the sum of dimensions $\dim V_s + \ldots +\dim V_t$ is finite, this process must terminate after a finite number of steps. 

\begin{lemma}\label{lem}
$V$ is indecomposable if and only if $V$ is $[-k,k]$-indecomposable for all non-negative integers $k$.
\end{lemma}
\begin{proof}
$\Leftarrow$: Assume that $V=U\oplus W$ for non-trivial $U$ and $W$. Then there exist indices $i_1$ and $i_2$ such that $U_{i_1}\neq 0$ and $W_{i_2}\neq 0$. This contradicts that $V$ is $[-\max(|i_1|, |i_2|), \max(|i_1|, |i_2|)]$-indecomposable. $\Rightarrow$: This follows by definition.  
\end{proof}

\subsection*{The Only Indecomposables are Interval Persistence Modules}
First we prove that every indecomposable is an interval, and then we show that every zigzag persistence module decomposes into a direct sum of indecomposables. 

The following result is well-known and can be found in many sources. For an elementary, self-contained proof, see \cite{ringel2}. 
\begin{theorem}\label{teo:An}
If $V$ is a finite zigzag persistence module, then $V$ decomposes as a finite direct sum of interval modules. 
\end{theorem}
Moreover, the theorem of Azumaya-Krull-Remak-Schmidt\cite{azumaya} asserts that such an interval decomposition is unique up to re-indexing.

Observe that if $V$ is an infinite zigzag persistence module, then the restriction of $V$ to an interval $[s,t]$ decomposes as a direct sum
\[ V|_{[s,t]} \cong\bigoplus_{i=1}^n I^{[a_i, b_i]}.\]

\begin{lemma}\label{lem:ext}
If there exists an interval $[s,t]$ such that the interval decomposition of $V|_{[s,t]}$ includes an interval  $I^{[a_j, b_j]}$ with $s<a_j$ and $b_j<t$, then $V$ is decomposable. 
\end{lemma}

\begin{proof}
To simplify notation we shall assume that both $s$ and $t$ are even, and that $a_j = s+1$ and $b_j = t-1$. The former can be achieved by increasing the interval $[s,t]$ and the latter is purely a cosmetic assumption to make the commutative diagram below smaller. Let 
\[ \phi: V|_{[s,t]} \to U\oplus I^{[a_j, b_j]}\]
be an isomorphism and let $\tilde{g}_{i}$ and $\tilde{f}_i$ denote the linear maps of the zigzag persistence module $U\oplus I^{[a_j, b_j]}$. This decomposition of $V|_{[s,t]}$ is extended to a decomposition of $V$ as described by the following commutative diagram
\[
\begin{tikzcd}
\cdots V_{s-1}\ar{d}{=}  & V_s\ar{l}[swap]{g_s}\ar{r}{f_s}\ar{d}{\cong}[swap]{\phi_s} & V_{s+1}\ar{d}{\cong}[swap]{\phi_{s+1}} & \cdots \ar{l}[swap]{g_{s+2}}\ar{r}{f_{s-2}} & V_{t-1}\ar{d}{\cong}[swap]{\phi_{t-1}} & V_t\ar{l}[swap]{g_t}\ar{r}{f_t}\ar{d}{\cong}[swap]{\phi_t} & V_{t+1}\cdots\ar{d}{=}\\
\cdots V_{s-1}  & U_s \ar{l}[swap]{\tilde{g}_s}\ar{r}{\tilde{f}_s}  & U_{s+1}\oplus I^{[a_j, b_j]}_{s+1} & \cdots \ar{l}[swap]{\tilde{g}_{s+2}}\ar{r}{\tilde{f}_{s-2}} & U_{t-1}\oplus I^{[a_j, b_j]}_{t-1} & U_t\ar{l}[swap]{\tilde{g}_t}\ar{r}{\tilde{f}_t} & V_{t+1}\cdots
\end{tikzcd}
\]
where we defined
\begin{align*}
\tilde{g}_s = g_s\circ \phi_s^{-1} & & \tilde{f}_t = f_t\circ \phi_t^{-1}
\end{align*}
and $\tilde{f}_i = f_i$, $\tilde{g}_i = g_i$ and $U_i = V_i$ for all $i\leq s-1$ and $i\geq t+1$. 
\end{proof}

\begin{lemma}
Let $V\colon \ZZCat\to \vect$ be indecomposable, then there exists a $t\geq 0$ such that $f_i$ is injective and $g_i$ is surjective for all $i\geq t$. Dually, there exists an $s\leq 0$ such that $f_i$ is surjective and $g_i$ is injective for all $i\leq s$. 
\end{lemma}
\begin{proof}
Let $\dim V_0 = N$; we shall show that there can be at most $N$ morphisms $f_i$, $i\geq 0$, that are non-injective. 

Assume that $\ker f_i \neq 0$ for an even integer $i\geq 0$ and look at the interval decomposition of $V|_{[0, i+1]}$. Since $\ker f_i \neq 0$, there exists an interval $I^{[a, i]}$ in the interval decomposition of $V|_{[0,i+1]}$. By indecomposability of $V$ and \cref{lem:ext} it follows that that $a = 0$. Thus, if there are $M$ indices $0\leq i_1< i_2 < \ldots < i_M$ such that $\ker f_{i_j} \neq 0$, then the interval decomposition of $V|_{[0, i_M+1]}$ has at least $M$ intervals supported on 0, implying that $M \leq \dim V_0 = N$. 

The setting with $g_i$ non-surjective is completely analogous.   
\end{proof}

\begin{lemma}\label{lem:injsurj}
Let $g_i: V_{i}\to V_{i-1}$ be a surjection and $f_i\colon V_i\to V_{i+1}$ an injection. If $V_{i-1}= U_{i-1}\oplus W_{i-1}$, then we can choose decompositions $V_i= U_i\oplus W_i$ and $V_{i+1}= U_{i+1}\oplus W_{i+1}$ such that 
\begin{align*}
g_i(U_{i}) &= U_{i-1} & g_i(W_{i}) &\subseteq W_{i-1}\\
f_i(U_i) &= U_{i+1} & f_i(W_i) &\subseteq W_{i+1}
\end{align*}
\end{lemma}
\begin{proof}
Define $U_i = g^{-1}(U_{i-1})$ and let $W_i$ be an internal complement of $U_i$ in $V_i$. From surjectivity of $g$ it follows that $g_i(U_i) = U_{i-1}$ and that $g_i(W_i)\subseteq W_{i-1}$. Similarly, define $U_{i+1} = f_i(U_i)$ and let $W_{i+1}$ be an internal complement of $U_{i+1}$ in $V_{i+1}$. Injectivity of $f_i$ implies $f_i(U_i) = U_{i+1}$ and $f_i(W_i) \subseteq W_{i+1}$. 
\end{proof}

We are now able to prove the first of our two needed results. 

\begin{theorem}
Let $V\colon \ZZCat\to \vect$ be indecomposable. Then $V$ is an interval persistence module. 
\end{theorem}

Let $m\geq 0$ be an index such that $f_i$ is injective and $g_i$ is surjective for all $i\geq m$. We shall show that $f_i$ and $g_i$ are isomorphisms. Assume for the sake of contradiction that $f_i$ is not surjective and decompose $V_{i+1} =  U_{i+1} \oplus \im f_i$ where $U_{i+1}$ is an internal complement of $\im f_i$. That yields the following sequence of vector spaces and linear maps:
\[
\begin{tikzcd}
\cdots V_i\ar{r}{f_i}\ar{d}{=} & V_{i+1}\ar{d}{=} & \ar{l}[swap]{g_{i+2}} V_{i+2}\ar{d}{=} \ar{r}{f_{i+2}} & V_{i+3}\ar{d}{=} & \ar{l}[swap]{g_{i+4}}\cdots\ar{d}{=}\\
\cdots V_i\ar{r}{f_i} & U_{i+1}\oplus \im f_{i} & \ar{l}[swap]{g_{i+2}} g_{i+2}^{-1}(U_{i+1})\oplus U_{i+2} \ar{r}{f_{i+2}} & f_{i+2}(g_{i+2}^{-1}(U_{i+1}))\oplus U_{i+3} & \ar{l}[swap]{g_{i+4}}\cdots
\end{tikzcd}
\]
where we have used \cref{lem:injsurj} together with the fact that $g_{i+2l}$ is surjective and $f_{i+2l}$ is injective for all $l\geq 1$. Since the process can be continued indefinitely, this contradicts that $V$ is indecomposable. Similarly, we must have that $g_{i}$ is an isomorphism for all $i\geq m$. Dually there exists an $m^\prime$ such that $f_i$ and $g_i$ are isomorphisms for all $i\leq m^\prime$. Hence, $V$ can be completely understood by its restriction to the interval $[m',m]$. The theorem follows by application of \cref{teo:An}. 

\subsection*{Decomposition into Intervals} The next thing we need to show is that every zigzag persistence module decomposes into a direct sum of interval modules. This is a special case of the first theorem in Section 6 of \cite{ringel}. 

\begin{theorem}\label{prop:decomp}
Any non-zero zigzag persistence module $V$ decomposes into a direct sum of interval persistence modules. 
\end{theorem}
\begin{proof}
We shall inductively define a $[-k, k]$-decomposition of $V$ for every $k\geq 0$. 

Start by choosing a $[0,0]$-decomposition of $V\cong \bigoplus_{j_0\in [m]} V^{(j_0)}$ where $[m] = \{0, \ldots, m\}$. The idea is to choose a $[-1,1]$-decomposition of $V^{(j_0)}$, and then, for every summand in the $[-1,1]$-decomposition of $V^{(j_0)}$, choose a $[-2,2]$-decomposition, and so forth. 
To illustrate the first step, let $j_0\in [m]$ be as above and let \[V^{(j_0)} \cong \bigoplus_{j_1\in [m_{j_0}]} V^{(j_0, j_1)}\] be a $[-1,1]$-decomposition of $V^{(j_0)}$. We parametrize the $[-1,1]$-indecomposables by a pair of indices $(j_0, j_1)$ under the convention that $V^{(j_0, j_1)}$ is the $j_1$-th $[-1,1]$-indecomposable in a $[-1,1]$-decomposition of $V^{(j_0)}$. Hence, the summands in a $[-2,2]$-decomposition of $V^{(j_0, j_1)}$ will be denoted by $V^{(j_0, j_1, j_2)}$ where $j_2\in [m_{(j_0, j_1)}]$, and so forth. 

Inductively, for every $(k+1)$-tuple $(j_0, \ldots, j_{k
})$ satisfying
\begin{equation}\label{eq}
j_i\in [m_{(j_0, \ldots, j_{i-1})}]\qquad \text{for all}\qquad  1\leq i\leq k,
\end{equation} choose a $[-(k+1), (k+1)]$-decomposition \[V^{(j_0, \ldots, j_{k})}\cong \bigoplus_{j_{k+1}\in [m_{(j_0, \ldots, j_{k})}]} V^{(j_0, \ldots, j_{k}, j_{k+1})},\]
which in turn yields a $[-(k+1), (k+1)]$-decomposition of $V$
\[
V\cong \bigoplus_{j_0\in [m]}\bigoplus_{j_{1}\in [m_{j_0}]}\cdots \bigoplus_{j_{k+1}\in [m_{(j_0, \ldots, j_{k})}]} V^{(j_0, \ldots, j_{k+1})} 
\]

Let $I$ be the set of all infinite sequences $s = (j_0, j_1, \ldots)$ such that the restriction to $(j_0, \ldots, j_k)$ satisfies \eqref{eq} for every $k$, and for every $s\in I$ define 
\[ V^s = V^{(j_0)}\cap V^{(j_0, j_1)}\cap V^{(j_0, j_1, j_2)}\cap\cdots.\] 

It is not hard to see that $V^s_i = V^{(j_0, \ldots, j_k)}_i$ for every $-k\leq i\leq k$. In particular, $V^s$ is $[-k,k]$-indecomposable for all $k\geq 0$ and thus indecomposable by \cref{lem}. Also, by the same observation, it follows that $V_i \cong \bigoplus_{s\in I} V^s_i$ for all $i$. 

Note that there can be sequences $s\in I$ such that $V^s = 0$. To give a proper direct sum decomposition of $V$ we let $I^\prime\subseteq I$ be the set of all sequences $s\in I$ such that $V^s \neq 0$. Hence, 
\[ V \cong \bigoplus_{s\in I^\prime} V^s.\]
\end{proof}
Since the interval modules have local endomorphism rings it follows from the theorem of Azumaya-Krull-Remak-Schmidt\cite{azumaya} that such a decomposition is unique up to permutation of the indexing set. 

To conclude this paper we provide an example showing that the assumption $\dim V_i<\infty$ for all $i\in \Z$ is crucial. The example is due to Michael Lesnick.
\begin{example}
Let $V^\prime =  \bigoplus_{k=1}^\infty I^{[-k, 0]}$ and define $V$ by the following properties: $V$ restricted to the non-positive integers equals $V^\prime$, $V_1= \mathbb{F}$, $V_i=0$ for $i>1$, and $f_0$ restricted to any of the $I^{[-k,0]}$ above is the identity. 

Assume that $V$ decomposes into a direct sum of interval persistence modules and let $I^{[a,1]}$ be the single interval summand that is non-zero at index $i=1$. Since $f_0\left(I^{[-k,0]}_0\right) = I^{[a,1]}_0$ for all $k\geq 1$ we must have that $I^{[a,1]}_{-k}$ is non-zero for all $k\geq 1$. Or, in other words, it must be of the form $I^{[-\infty,a]}$. This is not possible as the restriction of $V$ to non-positive integers equals a direct sum of interval persistence modules that are non-zero on a finite number of indices. Hence, there cannot be an interval persistence module containing $V_1$. 
\end{example}

\bibliographystyle{plain}
\bibliography{refs}

\begin{thebibliography}{1}

\bibitem{azumaya}
Gorô Azumaya.
\newblock Corrections and supplementaries to my paper concerning
  {K}rull-{R}emak-{S}chmidt's theorem.
\newblock {\em Nagoya Math. J.}, 1:117--124, 1950.

\bibitem{carlsson}
Gunnar~E. Carlsson and Vin de~Silva.
\newblock Zigzag persistence.
\newblock {\em Foundations of Computational Mathematics}, 10(4):367--405, 2010.

\bibitem{wcb}
William Crawley-Boevey.
\newblock Decomposition of pointwise finite-dimensional persistence modules.
\newblock {\em Journal of Algebra and Its Applications}, 14(05), 2015.

\bibitem{paq}
Charles Paquette.
\newblock Personal communication, 2015.

\bibitem{ringel}
Claus Ringel.
\newblock Introduction to representation theory of finite dimensional algebras.
\newblock \url{http://www.math.uni-bielefeld.de/\~ringel/lectures/izmir/}, 2014
  (accessed May 20 2015).

\bibitem{ringel2}
Claus Ringel.
\newblock The representations of quivers of type {$A_n$}. {A} fast approach.
\newblock \url{http://www.math.uni-bielefeld.de/\~ringel/opus/a_n.pdf},
  (accessed June 02 2015).

\bibitem{webb}
Cary Webb.
\newblock Decomposition of graded modules.
\newblock {\em Proceedings of the American Mathematical Society},
  94(04):565--571, 1985.

\end{thebibliography}

\end{document}